\documentclass[11pt]{article}
\usepackage[T1]{fontenc}
\usepackage[latin9]{inputenc}

\usepackage{epsfig}
\usepackage{amsmath,subfigure}
\usepackage{amssymb}
\usepackage{eepic}
\usepackage{latexsym
}
\usepackage{dsfont}

\usepackage[top=2 cm,bottom=2 cm,left=2. cm, right=2. cm]{geometry}

\usepackage{graphicx,color}
\usepackage{amsmath,amssymb,amsthm}

\theoremstyle{plain}
\newtheorem{theorem}{Theorem}
\theoremstyle{plain}

\theoremstyle{definition}
\newtheorem{definition}[theorem]{Definition}
\theoremstyle{plain}

\theoremstyle{plain}
\newtheorem{assumption}[theorem]{Assumption}
\theoremstyle{plain}

\theoremstyle{remark}
\newtheorem*{rem*}{Remark}

\def\T{{\mathcal T}}
\def\ve{{\varepsilon}}

\begin{document}

\title{Stochastic homogenization of the Keller-Segel chemotaxis system}

\author{{Anastasios Matzavinos$\,^{\mathrm{a,b}}$ and   Mariya Ptashnyk$\,^{\mathrm{c}}$}\medskip\\
\small $^{\mathrm{a}}$ Division of Applied Mathematics, Brown University, Providence, RI 02912, USA \vspace*{-0.1cm}\\
\small $^{\mathrm{b}}$ Computational Science and Engineering Laboratory, ETH Z\"{u}rich, CH-8092, Z\"{u}rich, Switzerland \vspace*{-0.1cm}\\
\small $^{\mathrm{c}}$ Division of Mathematics, University of Dundee, 
Dundee, DD1 4HN, UK}
\date{}
\maketitle

\begin{abstract}
In this paper, we focus on the Keller-Segel chemotaxis system in a random heterogeneous domain. We assume that the corresponding diffusion and chemotaxis coefficients are given by stationary ergodic random fields and apply stochastic two-scale convergence methods to derive the homogenized macroscopic equations. In establishing our results, we also derive {\it a priori} estimates for the Keller-Segel system that rely only on the boundedness of the coefficients; in particular, no differentiability assumption on the diffusion and chemotaxis coefficients for the chemotactic species is required.  Finally, we prove the convergence of a periodization procedure for approximating  the homogenized macroscopic coefficients. \\

\end{abstract}


\section{Introduction} 
Chemotaxis as a term refers to the directed movement of cells and microorganisms in response to a chemical signal. Historically, the first mathematical model of chemotaxis was proposed by Keller and Segel in order to investigate the aggregation dynamics of cellular slime molds, such as the social amoeba \textit{Dictyostelium discoideum} \cite{KS71}. Since then, the Keller-Segel model has been analyzed extensively, and  a comprehensive review of  related mathematical results  can be found in the two articles by Horstmann \cite{Horstmann2003, Horstmann2004}.

It is well known that in one dimension the Keller-Segel model is well-posed globally in time.  Global existence and boundedness of solutions in one dimension were first shown by Yagi \cite{Yagi1997} by means of energy estimates. Moreover,  the well-posedness and the existence of a finite-dimensional attractor for the one-dimensional model was proved by Osaki and Yagi  \cite{OY2001}.  
      
The dynamics of the Keller-Segel model in two and three dimensions are more complex than the one-dimensional case, since in higher dimensions the solutions may blow up in finite time \cite{Jaeger1992, Nagai1995,Perthame,Winkler2011}. Several results that appeared in the 1990's have demonstrated that in two and three dimensions the Keller-Segel model is well-posed globally in time for ``small'' initial data. However,  in the presence of  ``large'' initial data, the solutions blow up; in other words, they do not remain bounded \cite{ Horstmann2001, Horstmann2005, Nagai1997, Yagi1997}.

 Corrias and Perthame \cite{CorriasPerthame2006} showed that in $d$ dimensions, the Keller-Segel model is critical in $L^{d/2}$, which is to say that the ``smallness'' or ``largeness'' of the initial data is determined in terms of the $L^{d/2}$ norm. Similar conditions were derived in  \cite{CorriasPerthame2004, CorriasPerthameZaag2004} for a parabolic-elliptic variation of the Keller-Segel model. The global behavior of a two-dimensional parabolic-parabolic chemotaxis system, under the assumption of ``small'' initial data, was investigated by Gajewski and  Zacharias \cite{Gajewski1998}.
 
As alluded to in the above paragraphs, there is a wealth of results on the existence and regularity of solutions of the Keller-Segel model. However, there is no literature investigating homogenization approaches and the influence of substrate heterogeneity  on the dynamics of the  model. 

  Stochastic homogenization is a growing  field in multiscale analysis. Some of the first  results on the stochastic homogenization of linear second-order elliptic equations  were obtained by Kozlov \cite{Kozlov1980} (by a direct contraction of the corrector functions), by Papanicolaou and Varadhan \cite{Papanicolaou1979} (by using Tartar's energy method), and by Zhikov {\it et al.} \cite{ZKON1979} (by using $G$-convergence of operators). Subsequently, the homogenization of quasi-linear elliptic  and parabolic equations with stochastic coefficients was considered by Bensoussan and Blankenship \cite{Bensoussan1988} and  Castell \cite{Castell2001}. The stochastic homogenization of convex integral operators by means of $\Gamma$-convergence was considered by Dal Maso and  Modica  \cite{DalMaso1986_1, DalMaso1986_2}. The method of  viscosity solutions was employed by Caffarelli {\it et al.}  \cite{Caffarelli2005} to derive effective equations for fully nonlinear  elliptic and parabolic equations in stationary ergodic media. In a similar fashion, subadditive ergodic theory has been used together with the theory of viscosity solutions or variational representations of solutions and the minimax theorem  to homogenize  Hamilton-Jacobi and viscous Hamilton-Jacobi equations in stationary ergodic media \cite{Armstrong2012, Kosygina2006, Souganidis2005, Souganidis2010} (see also references therein). More recently, subadditive ergodic theory has also been employed to homogenize quasiconvex (level-set convex) and, more generally, non-convex  Hamilton-Jacobi equations in stationary ergodic media \cite{Armstrong2013, Armstrong2015}. 
   
  The theory of  periodic two-scale convergence  \cite{Allaire1992, Lukkassen2002, Nguetseng1989} has been extended in the stochastic setting by Bourgeat, Mikeli\'c, and Wright \cite{BMW1994}, who defined the concept of two-scale convergence in the mean, and by Zhikov and Piatnitski  \cite{Zhikov_Piatnitsky_2006}, who defined an explicitly stochastic two-scale convergence for random measures. 
 The  two-scale convergence in the mean has been applied to derive macroscopic equations for  single- and two-phase fluid flows in randomly fissured media  \cite{BMP2003, Wright2001}. The stochastic two-scale convergence has been  extended to Riemannian manifolds and has been applied to analyze heat transfer through composite and polycrystalline  materials with nonlinear conductivities \cite{Heida_2012, Heida_2011}. 

The paper is organized as follows. In section 2, we  formulate a microscopic chemotaxis model with 
diffusion and chemotaxis coefficients  for the chemotactic species given by stationary ergodic random fields. In contrast, and consistent with the experimental setting discussed in section 2, the diffusion coefficient of the chemical species (chemoattractant) is assumed to be deterministic, i.e., independent of the random medium. We  then 
derive {\it a priori}  estimates in section 3 and prove the existence and uniqueness of  weak solutions for the microscopic model. 
Our derivation of the  {\it a priori} estimates differs from those found in \cite{Hillen2004}, \cite{Nagai1997}, or \cite{OY2001}, as we only assume the boundedness of the rapidly oscillating   coefficients describing the stochastic medium. In section 4, we use the derived {\it a priori} estimates and the notion of stochastic two-scale convergence to derive a macroscopic (homogenized) model for our system. Two auxiliary stochastic problems are obtained to define the macroscopic diffusion  and  chemosensitivity coefficients for the chemotactic species. In section 5, we use a periodization procedure and prove the convergence of the effective coefficients obtained by periodic approximation to the corresponding macroscopic coefficients obtained by the stochastic homogenization approach of section 4.

\section{Formulation of the problem}\label{section222}
We consider a variation of the original Keller-Segel  model of chemotaxis \cite{KS71}, where the coefficients of the model are defined by stationary random fields. Specifically, we consider the system:  
\begin{equation}\label{micro_model}
\begin{aligned} 
u^\ve_t &=&\nabla\cdot (D_u^\ve(x)\nabla u^\ve-\chi^\ve(x) u^\ve\nabla v^\ve),\;\;\quad  x\in Q,\, t>0,\\
v^\ve_t &=& \nabla\cdot(D_v(x)\nabla v^\ve)-\gamma v^\ve +\alpha u^\ve, \;\;\quad  x\in Q,\, t>0, \\
&& \frac{\partial u^\ve}{\partial n} =0, \quad \frac{\partial v^\ve}{\partial n}=0,  \quad x\in\partial Q,\, t>0,  \\
&& u^\ve(0,x) = u_0(x), \; \quad  v^\ve(0,x)=v_0(x), \; \quad  x\in Q, 
\end{aligned} 
\end{equation} 
where $Q\subset \mathbb R^d$ is a bounded domain and $\alpha$, $\gamma$ are positive constants. Moreover, $u^\ve$ and $v^\ve$ denote the density of a population of cells (the chemotactic species) and the concentration of a chemoattractant, respectively.

As will become apparent in the following, the parameter $\ve$ represents the spatial scale of the microscopic structure of the underlying medium or substrate.
The diffusion coefficient $D_u^\ve$ and the chemosensitivity function $\chi^\ve$ depend on $\ve$, as they are affected by changes in the properties of the substrate. It is assumed that these changes do not affect the diffusion of chemicals, and specifically the diffusion coefficient $D_v$ does not depend on $ \ve$ (nonetheless, we allow for $D_v$ to be a smooth enough function of the spatial variable $x$). This is consistent with {\it in vitro} experiments where the cells are positioned on a micropatterned surface, and hence their random and chemotactic motility are affected by the microstructure, whereas the chemoattractant diffuses freely in the solution above the surface \cite{Gray2003}.

In order to specify the dependence of the model coefficients on the microscopic scale $\ve$, we introduce the concept of a spatial dynamical system as follows (see, e.g., \cite{BMW1994}). We consider a probability space $(\Omega, \mathcal F, P)$ with probability measure $P$.   Throughout the paper, $\Omega$ is assumed to be a compact metric space and $\mathcal F$ is the $\sigma$-algebra of Borel sets over $\Omega$. 
We define  a  spatial dynamical system $\T(x):\Omega \to \Omega$, i.e.\  a family $\{\T(x)\, : \, x\in\mathbb{R}^d\}$ of invertible maps, such that for each $x\in \mathbb R^d$, both $\T(x)$ and $\T^{-1}(x)$ are measurable and satisfy the following conditions:
\begin{enumerate}
\item[(i)] $\T(0)$  is the identity map on  $ \Omega$ and  $\T(x)$ satisfies the semigroup property:
$$ \T(x_1+x_2)= \T(x_1)\T(x_2) \quad \text{ for all }  x_1, \, x_2 \in \mathbb R^d. $$
\item[(ii)]  $P$ is an invariant measure for $\T(x)$, i.e.\ for each $x\in \mathbb R^d$ and $F\in \mathcal F$ we have that $$P(\T^{-1}(x) F) = P(F).$$ 
\item[(iii)] For each $F\in \mathcal F$,  the set $\{ (x,\omega)\in \mathbb R^d\times \Omega : \T(x) \omega \in F\} $ is a $dx\times d P(\omega)$-measurable subset of $\mathbb R^d \times \Omega$, where $dx$ denotes the Lebesgue measure on $\mathbb R^d$. 
\end{enumerate}

The coefficients in (\ref{micro_model}) are defined as follows. First, we define two stationary random fields through the relations
$$D_u(x,\omega) =\widetilde D_u(\T(x)\omega) \,\mbox{ and }\, \chi(x, \omega)=\widetilde \chi(\T(x)\omega),   \label{stat}$$
where $\widetilde D_u$ and $\widetilde \chi$ are given measurable functions over $\Omega$. Then, given the specified assumptions on the random fields, the coefficients $D_u^\ve(x)$ and $\chi^\ve(x)$ are defined as 
$$D_u^\ve(x) = D_u(x/\ve,  \omega) \,\mbox{ and }\, \chi^\ve(x)= \chi(x/\ve,  \omega).$$
From a mathematical point of view, this construction of the coefficients is common in the stochastic homogenization literature because it allows for the use of ergodic theory in the asymptotic investigation of \eqref{micro_model} as $\ve\rightarrow 0$ (see section 4). From a modeling perspective, this construction is equivalent to the assumption that the coefficients are statistically homogeneous (see, e.g., \cite{Daley2008}). As alluded to above, the chemoattractant diffusion coefficient $D_v$ does not depend on $\ve$. 

As an example, we discuss here a specific construction of $(\Omega,\mathcal{F},P)$ and  $\T(x)$ based on the Poisson point process in order to provide some intuition on  the abstract setting discussed above. Consider the case where motile cells are positioned on a micropatterned surface with randomly imprinted ``dots,'' i.e.   $D_u^\ve(x)$ and $\chi^\ve(x)$ are assumed to attain distinct  values in the union of randomly dispersed balls  and in their exterior. Then, a realization $\omega\in\Omega$ is identified with a set $\omega=\{B(\boldsymbol{\alpha}_m)\,:\, m\in\mathbb{N}\}$ of a spatial distribution of balls $B(\boldsymbol{\alpha}_m)$ of a specified radius centered at $\boldsymbol{\alpha}_m$, and  the $\sigma$-algebra $\mathcal{F}$ is defined as follows. Let $N(\omega, A)$ denote the number of balls the centers of which fall in the open set $A\subset\mathbb{R}^2$. Then, $\mathcal{F}$ is the $\sigma$-algebra generated  by the subsets of $\Omega$ of the form 
 $$\{\omega\in\Omega \,:\, N(\omega, A_1)=k_1,\ldots, N(\omega,A_i)=k_i\},$$
 where $i, k_1,\ldots, k_i$ are non-negative integers and $A_1, \ldots, A_i$ are disjoint open sets. 
 A natural choice for the probability measure $P$ (in the absence of any {\it a priori} information) is given by the Poisson point process defined in the following way. We let
 $$P\bigl(N(\omega, A_1)=k_1, \ldots, N(\omega, A_i)=k_i\bigr) \,= \,P\bigl(N(\omega, A_1)=k_1\bigr)\times \ldots\times P\bigl( N(\omega, A_i)=k_i\bigr),$$
with 
 $$P\bigl(N(\omega,A)=k\bigr)\, = \, \frac{(\lambda |A|)^k}{k!}\exp(-\lambda|A|),$$
where $\lambda$ is a positive parameter. In this setting,  $\T(x)$ is defined as the family of translation operators given by:
$$\T(x)\,\omega = \{B(\boldsymbol{\alpha}_m)+x\, : \, m\in\mathbb{N}\},$$
where $x\in\mathbb{R}^2$ and $\omega=\{B(\boldsymbol{\alpha}_m)\,:\, m\in\mathbb{N}\}$. One can define a metric that turns $\Omega$ into a compact metric space, as required in the more general setting of section  \ref{section222}. This can be achieved either by considering an alternative characterization of the Poisson point process as a point process over {\it i.i.d.} compact domains that cover the Euclidean space (see, e.g., \cite{Daley2008}) or by using an appropriate weighting and normalization of one of the standard sequence space norms (see, e.g., \cite{Kechris}). This specific construction of $(\Omega,\mathcal{F},P)$  and $\mathcal T(x)$ is intuitive from a modeling perspective. Nonetheless, the somewhat more abstract setting of a  spatial dynamical system is quite versatile, and will be adopted in the remainder of the paper.

The following assumption is used throughout the paper. 

\begin{assumption}\label{assumptions}
The following hold:
\begin{enumerate}
\item[(i)] It is assumed that 
$0< d^0_u\leq \widetilde D_u(\omega)\leq d^1_u < \infty$   and 
 $0\leq \widetilde \chi(\omega)\leq \chi^1 < \infty$  for $ P$-a.s.\ $\omega \in  \Omega$.
\item[(ii)] It is assumed that $D_v \in  W^{2,\infty}( Q)$ is strongly elliptic, i.e.,    $$0< d_v^0 \leq (D_v(x)\xi, \xi)  \leq d_v^1< \infty\text{ for }x \in Q\text{ and }\xi \in \mathbb R^d,$$ and  $\sup\limits_{Q} |\nabla D_v (x) |+ \sup\limits_{Q} |\nabla^2 D_v (x) | \leq d_v^2$,  and 
$\alpha$, $\gamma$ are positive constants. 
 \item[(iii)] With respect to the initial conditions, it is assumed that 
$$u_0 \in H^1(Q), \, v_0 \in H^2(Q),\text{ and } u_0(x)\geq 0,\, v_0(x)\geq 0\,   \mbox { for a.e. }  x\in Q .$$
  Moreover, if $d=\dim (Q)= 2$ or $d=3$, it is additionally assumed that
\begin{equation}\label{estim_assum_in}
\begin{aligned}
\Big(1+ |Q|^{\frac {2-r} 2}\|u_0\|_{L^1(Q)}^{r/2} \Big) \Big[\max\big\{\| u_0\|_{L^r(Q)}, C_g\big(\| u_0\|_{L^1(Q)}+ \| u_0\|^{\frac  2{d+2}}_{L^1(Q)} \big)\big\} \\ +\|\nabla v_0\|_{L^q(Q)} \Big]   < \frac{2 d_u^0}r \frac 1{ \chi^1 C_b C_v},
\end{aligned}
\end{equation}
where $q=\max\{2+\zeta, d\}$,  $1+ \frac \zeta{4+\zeta}< r\leq 2$ for any $\zeta>0$ if $d=2$, and 
$\frac d 2 <r \leq 2$ if $d = 3$. The constants 
 $C_v$,  $C_b$, and $C_g$ appear in the estimates  \eqref{estim_grad_v} and \eqref{estim_d2_emb}--\eqref{estim_d3_GN}.   
\end{enumerate}
\end{assumption}

We are now in a position to define the concept of weak solution that is used in this paper. In the following, $Q_\tau= (0,\tau)\times Q$ for  $\tau>0$, and $\langle\,\cdot,\cdot\,\rangle_{Q_\tau}$ denotes the integral  $\langle u,v\rangle_{Q_\tau}=\int_0^\tau\int_Q uv\,dxdt$.

\begin{definition}
The pair $(u^\ve,v^\ve)$  is a weak solution of  \eqref{micro_model} if 
 $u^\ve \in  L^2(0, \tau; H^1(Q))\cap H^1(0, \tau; L^2(Q))$,  $v^\ve \in   L^4(0,\tau; W^{1,4}(Q))\cap H^1(0, \tau; L^2(Q))$, 
 and 
 \begin{eqnarray}
&&\langle u^\ve_t, \phi\rangle_{Q_\tau} 
+\langle D^\ve_u(x) \nabla u^\ve - \chi^\ve(x) u^\ve \nabla v^\ve , \nabla \phi \rangle_{Q_\tau} = 0 ,  \label{micro_u} \\
&&\langle v^\ve_t, \psi\rangle_{Q_\tau} 
+\langle D_v(x)\nabla v^\ve , \nabla \psi \rangle_{Q_\tau}  + \gamma \langle v^\ve, \psi\rangle_{Q_\tau} = \alpha \langle u^\ve, \psi\rangle_{Q_\tau},  \label{micro_v}
\end{eqnarray}
for any $\phi, \psi \in L^2(0,\tau; H^1(Q))$ and $P$-a.s.\ in $\Omega$. Moreover, $u^\ve$ and $v^\ve$  satisfy the initial conditions $u^\ve(0,x)= u_0(x)$,  $v^\ve(0,x) =v_0(x)$ in $L^2(Q)$  for  $P$-a.s. $\omega\in\Omega$. 
\end{definition}

\section{Existence of solutions of the microscopic problem and a priori estimates} 

In this section, we establish a priori estimates for the weak solutions of (\ref{micro_model}) that eventually lead to the proof of our main homogenization result in section 4.  In what follows, we distinguish (and treat differently) the cases $\text{dim} (Q) =1$ and $\text{dim} (Q) \geq 2$.  In the latter case, motivated by experimental and modeling settings for  biological and physical systems,  we only consider the cases  $\text{dim}(Q)=2$ and $\text{dim}(Q)=3$. However similar results can also be obtained when $\text{dim} (Q)\geq 4$.

If $\text{dim} (Q) =1$, the chemotaxis system has a global solution as shown in  \cite{Hillen2004, OY2001,  Yagi1997}.  However, since the system studied in this paper has fast oscillating diffusion and chemotaxis coefficients, we provide a different proof of the  well-posedness  of the system than the one developed in \cite{Hillen2004, OY2001,  Yagi1997}. Specifically, our derivation of the {\it a priori} estimates does not require the differentiability of $D^\ve_u$ or $\chi^\ve$. 

\begin{theorem}\label{thm2}
Under Assumption~\ref{assumptions}  and $\text{dim} (Q) =1$ there exists a unique weak  solution of  \eqref{micro_model} for every $\ve >0$, and for $P$-a.s.\ $\omega\in\Omega$  we have
\begin{eqnarray}\label{apriori_estim}
\begin{aligned}
&\|u^\ve\|_{L^\infty(0,\tau; L^2(Q))}  +
 \|\partial_x u^\ve\|_{L^\infty(0,\tau; L^2(Q))}+  \|\partial_t u^\ve\|_{L^2(Q_\tau)} \;  \leq C, \\
&\| v^\ve\|_{L^\infty(0,\tau; H^1(Q))} +  \| \partial_t  v^\ve\|_{L^2(0,\tau;H^1(Q))}
+  \|\partial^2_x v^\ve\|_{L^\infty(0,\tau; L^2(Q))}\;    \leq C, 
\end{aligned}
\end{eqnarray}
for  any $\tau >0$, where the constant $C$  is independent of $\ve$.
\end{theorem} 
\begin{proof}
The existence of a weak solution to problem \eqref{micro_model} is proved by showing the existence of a fix point of the operator $K$ defined on $L^4(0,\tau; W^{1,4}(Q))$ by $v^{\ve} = K(\overline v^{\ve})$ with  $v^{\ve}$ given as a solution of the linear problem 
\begin{equation}\label{micro_model_FP}
\begin{aligned} 
&u^{\ve}_t =&&\partial_x\cdot (D_u^{\ve}(x)\, \partial_x u^{\ve}-\chi^\ve(x)\,  u^{\ve}\, \partial_x \overline v^{\ve})\; && \quad \text{ in } Q_\tau  \; ,\\
&v^{\ve}_t =& &\partial_x\cdot(D_v(x)\, \partial_x v^{\ve})-\gamma\,  v^{\ve} +\alpha\,  u^{\ve} \; &&\quad \text{ in  } Q_\tau \; , \\
&&& \partial_x u^{\ve} =0, \qquad \; \qquad\partial_x v^{\ve}=0  &&\quad  \text{ on }  (0,\tau)\times \partial Q\ \; ,  \\
&&& u^{\ve}(0,x) = u_0(x), \; \quad  v^{\ve}(0,x)=v_0(x) \; && \quad  \text{ in } Q \; .
\end{aligned} 
\end{equation} 
By applying Galerkin's method \cite{Evans2010} and \textit{a priori} estimates similar to the estimates \eqref{estim_v}, \eqref{estim_u_1},  \eqref{estim_v_2}, and  \eqref{estim_u_3} established below, we obtain for every  $ \overline v^{\ve} \in L^4(0,\tau; W^{1,4}(Q))$ the existence of solutions $(u^{\ve}, v^{\ve})$ of \eqref{micro_model_FP} with $u^{\ve} \in  L^2(0, \tau; H^1(Q))\cap H^1(0, \tau; L^2(Q))$ and $v^{\ve} \in H^1(0, \tau; L^2(Q))\cap L^\infty(0,\tau; H^2(Q))$.  Then, the compact embedding  $L^4(0,\tau; H^2(Q)) \cap H^1(0, \tau; L^2(Q)) \subset L^4(0,\tau; W^{1,4}(Q))$, along with the Schauder Fixed point theorem and \textit{a priori} estimates    ensure  the existence of a solution to the original nonlinear  problem~\eqref{micro_model} for all $\ve >0$. 
 The regularity of the solutions ensures that $u^\ve, \, v^\ve \in C([0,\tau]; L^2(Q))$  for $P$-a.s.\   $\omega \in \Omega$, and thus the initial conditions are satisfied. 

We  also remark that the  {\it a priori} estimates are first derived for Galerkin approximations constructed by  smooth eigenfunctions of the one-dimensional Laplace operator with Neumann boundary conditions. Then, using standard arguments pertaining to the weak convergence and lower semicontinuity of the norms involved, we also obtain the corresponding estimates for the solutions $u^\ve$ and $v^\ve$ of \eqref{micro_model}.

We remark that, provided Assumption~\ref{assumptions}, the solutions of (\ref{micro_model}) remain nonnegative for all times, see e.g.\  \cite{OY2001,Perthame}. To prove the required {\it a priori} estimates, we first consider  $\phi=1$ and $\psi=1$ as  test functions in  \eqref{micro_u} and \eqref{micro_v} to obtain
\begin{equation}\label{L1_estim}
\|u^\ve(t)\|_{L^1(Q)} = \|u_0\|_{L^1(Q)} \qquad \text{ for } t\geq 0  \; , 
\end{equation}
and 
\begin{equation}\label{L1_estim2}
\partial_t \|v^\ve(t)\|_{L^1(Q)}  =    - \gamma  \|v^\ve(t)\|_{L^1(Q)} + \alpha \|u^\ve(t)\|_{L^1(Q)}  \quad \text{ for } t > 0\; . 
\end{equation}
Hence, we  obtain 
\begin{equation}
\|v^\ve(t)\|_{L^1(Q)} =  \|v_0\| _{L^1(Q)} e^{-\gamma t} +  \alpha\gamma^{-1}(1 - e^{-\gamma t}) \|u_0\|_{L^1(Q)}\; \quad \text{ for } t\geq 0\; .
\end{equation}
Multiplying the second equation in \eqref{micro_model}
by   $ v^\ve$ and  $\partial^2_{x} v^\ve$, integrating over $Q$,  and using zero-flux boundary conditions together with the specified assumptions on $D_v$,  we have
\begin{eqnarray*}
\frac 12\partial_t \| v^\ve(t)\|^2_{L^2(Q)} + d_v^0\|\partial_x v^\ve(t)\|^2_{L^2(Q)}  + \gamma \|v^\ve(t) \|^2_{L^2(Q)}  \leq   \alpha \|u^\ve(t)\|_{L^2(Q)}\|v^\ve(t)\|_{L^2(Q)}, 
\\
\frac 12 \partial_t \|\partial_x v^\ve(t)\|^2_{L^2(Q)} + d_v^0 \|\partial^2_x v^\ve(t)\|^2_{L^2(Q)}  + \gamma \|\partial_x v^\ve (t)\|^2_{L^2(Q)} \hspace{3.8 cm } \\  \leq  \alpha \|u^\ve(t)\|_{L^2(Q)}\|\partial_x^2  v^\ve(t)\|_{L^2(Q)} 
+ \, d_v^2 \, \|\partial_x v^\ve(t) \|_{L^2(Q)}\|\partial_x^2  v^\ve(t)\|_{L^2(Q)} \; . 
\end{eqnarray*}
Applying Young's and Gronwall's inequalities and using $v_0 \in H^1(Q)$ yield
\begin{eqnarray}\label{estim_v}
\| v^\ve\|_{L^\infty(0,\tau; L^2(Q))} + \|\partial_x v^\ve\|_{L^\infty(0,\tau; L^2(Q))} + 
 \|\partial^2_x v^\ve\|_{L^2(Q_\tau)}\hspace{2 cm }   \nonumber \\ \leq  C_1   \|u^\ve\|_{L^2(Q_\tau)}  + C_2\; ,
\end{eqnarray}
where the constants $C_1$ and $C_2$ are independent of $\ve$.

Multiplying the first equation in \eqref{micro_model}
by  $u^\ve$,  integrating over $Q$, and using zero-flux boundary conditions together with the stated assumptions on $\widetilde D_u$  give 
\begin{eqnarray*}
\partial_t \| u^\ve(t)\|^2_{L^2(Q)} + 2 d_u^0\, \|\partial_x u^\ve(t)\|^2_{L^2(Q)} 
\leq 2\, \langle \chi^\ve(x) u^\ve(t)\, \partial_x v^\ve(t), \partial_x u^\ve(t) \rangle_{Q}\; .
\end{eqnarray*}
The term on the right-hand side can be estimated as 
\begin{eqnarray*}
\langle \chi^\ve(x) u^\ve \partial_x v^\ve, \partial_x u^\ve \rangle_{Q} \leq \frac{(\chi^1)^2}{d_u^0} \|u^\ve \partial_x v^\ve\|^2_{L^2(Q)} +
\frac{d_u^0} 4\|\partial_x u^\ve \|^2_{L^2(Q)} \; . 
\end{eqnarray*}
 We  use the Gagliardo-Nirenberg inequality,  i.e.\ for $w \in W^{1,l}(Q)$ we use
\begin{equation}\label{GN_In}
\|w\|_{L^s(Q)} \leq \tilde C\left( \|\nabla w\|_{L^l(Q)}^\sigma \|w\|_{L^q(Q)}^{1-\sigma} + \|w\|_{L^1(Q)} \right),  \;  \quad \frac 1  s = \sigma\Big[\frac 1 l - \frac  1d \Big] + (1-\sigma) \frac 1 q, 
\end{equation}
with (a) $d=\text{dim}(Q)=1$, $s=4$,  $\sigma =1/2$, $l=2$, $q=1$,  (b) $d=1$, $s=2$, $\sigma = 1/3$, $l=2$, $q=1$, and (c)  $d=1$, $s=4$, $\sigma = 1/4$, $l=2$, $q=2$, respectively,  to obtain
 \begin{eqnarray}
 \|u^\ve\|_{L^4(Q)} &\leq & \tilde C\big(\|\partial_x u^\ve \|^{1/2}_{L^2(Q)} \|u^\ve\|^{1/2}_{L^1(Q)} +  \|u^\ve\|_{L^1(Q)} \big)\; ,\label{u_GN} \\
  \|u^\ve\|_{L^2(Q)} & \leq & \tilde C\big(\|\partial_x u^\ve \|^{1/3}_{L^2(Q)} \|u^\ve\|^{2/3}_{L^1(Q)} +  \|u^\ve\|_{L^1(Q)} \big)\; , \label{u_2_GN}\\
 \|\partial_x v^\ve\|_{L^4(Q)} & \leq &  \tilde C\big(\|\partial^2_x v^\ve \|^{1/4}_{L^2(Q)} \|\partial_x v^\ve\|^{3/4}_{L^2(Q)} +  \|\partial_x  v^\ve\|_{L^2(Q)} \big)\; .  \label{v_GN} 
 \end{eqnarray}
Thus, using estimate \eqref{v_GN} we have 
\begin{eqnarray*}
 \int_0^\tau \|\partial_x v^\ve\|^4_{L^4(Q)}  dt &\leq  & 8 \tilde C \int_0^\tau \big[\|\partial^2_x v^\ve\|_{L^2(Q)} 
\|\partial_x v^\ve\|^3_{L^2(Q)} +  \|\partial_x v^\ve\|^4_{L^2(Q)}\big] dt  \\
&\leq& C\big[\sup_{(0,\tau)} \|\partial_x v^\ve\|^3_{L^2(Q)}   \|\partial^2_x v^\ve\|_{L^2(Q_\tau)} 
+\sup_{(0,\tau)} \|\partial_x v^\ve\|^4_{L^2(Q)} \big]\; . 
 \end{eqnarray*}
 Then, the estimate in \eqref{estim_v}  together with \eqref{u_2_GN} ensure that
\begin{eqnarray*}
  \|\partial_x v^\ve\|^4_{L^4(Q_\tau)}
&\leq& C_1  \big[\|u^\ve\|^{4}_{L^2(Q_\tau)} +1\big]  \\
& \leq&  C_2\big[\|\partial_x u^\ve \|^{4/3}_{L^2(Q_\tau)} \sup_{(0,\tau)}\|u^\ve\|^{8/3}_{L^1(Q)} + \sup_{(0,\tau)}\|u^\ve\|^{4}_{L^1(Q)} +1\big] \; .
 \end{eqnarray*}
Hence, using the last  inequality along with \eqref{L1_estim} and \eqref{u_GN}, we obtain 
\begin{eqnarray*}
&& \frac{(\chi^1)^2}{d_u^0} \|u^\ve \partial_x v^\ve\|^2_{L^2(Q_\tau)}  \leq 
\frac {d_u^0} { 8 \tilde C(\|u_0\|^2_{L^1(Q)}+1)} \|u^\ve\|^4_{L^4(Q_\tau)}  + C_1 \|\partial_x v ^\ve\|^4_{L^4(Q_\tau)} \\ && \hspace{ 1.5 cm} 
\leq  \frac {d_u^0} 8 \|\partial_x u^\ve \|^{2}_{L^2(Q_\tau)} +  C_2  \|\partial_x u^\ve \|^{4/3}_{L^2(Q_\tau)}  + C_3
\leq   \frac {d_u^0} 4 \|\partial_x u^\ve \|^{2}_{L^2(Q_\tau)} + C_4\; . 
\end{eqnarray*}
Combining all estimates together,  we have that 
\begin{eqnarray}\label{estim_u_1}
\|u^\ve\|_{L^\infty(0,\tau; L^2(Q))}  +
\|\partial_x u^\ve\|_{L^2(Q_\tau)}   \leq C \; ,
\end{eqnarray}
where the constant $C$ is  independent of $\ve$. Using $\text{dim}(Q)=1$ in the last estimate, we obtain that
\begin{eqnarray}\label{estim_u_2}
  \|u^\ve\|_{L^2(0,\tau; L^\infty(Q))}  \leq C \; . 
  \end{eqnarray}
Considering  $\partial_t \partial^2_{x} v^\ve$  as a test function in \eqref{micro_v}, applying integration by parts, and using zero-flux boundary conditions  together with the specified assumptions on $D_v$ yield that 
\begin{eqnarray*}
 \int\limits_0^\tau\Big[ \|\partial_t  \partial_xv^\ve\|^2_{L^2(Q)} +\frac {d_v^0}2 \partial_t  \|\partial^2_x v^\ve\|^2_{L^2(Q)}  + \frac \gamma 2 \partial_t \|\partial_x v \|^2_{L^2(Q)} \Big] dt \leq \alpha\int\limits_0^\tau  | \langle \partial_x u^\ve, \partial_t   \partial_x v^\ve \rangle_{Q} | dt    \\
 + 
\int\limits_0^\tau \big|\langle  \partial^2_x D_v(x) \partial_x v^\ve + \partial_x D_v(x) \partial^2_x v^\ve, \partial_t\partial_x v^\ve \rangle_{Q} \big| dt
\\
 \leq  \frac 14  \int\limits_0^\tau \|\partial_t \partial_x  v^\ve\|^2_{L^2(Q)} dt
  + C    \int\limits_0^\tau \big[ \|\partial_x u^\ve\|^2_{L^2(Q)}+\|\partial^2_x v^\ve\|^2_{L^2(Q)}
 +  \|\partial_x v^\ve\|^2_{L^2(Q)} \big]dt \; , 
\end{eqnarray*}
where $C= C(d_v^2, \alpha)$. 
Then using \eqref{estim_v}, \eqref{estim_u_1}, and the assumption $v_0 \in H^2(Q)$, we have
\begin{eqnarray}\label{estim_v_2}
 \| \partial_t \partial_x  v^\ve\|_{L^2(Q_\tau)} +  \|\partial^2_x v^\ve\|_{L^\infty(0,\tau; L^2(Q))} 
 +  \|\partial_x v \|_{L^\infty(0,\tau; L^2(Q))}  \leq C \; .
 \end{eqnarray}
Multiplying the first equation in \eqref{micro_model}
by  $u^\ve_t$, integrating over $Q$ and using zero-flux boundary conditions  we obtain 
\begin{eqnarray*}
\|\partial_t u^\ve(t)\|^2_{L^2(Q)}+ \langle  D_u^\ve(x)\,  \partial_x u^\ve(t), \partial_t \partial_x  u^\ve(t) \rangle_Q = 
\langle \chi^\ve(x) \, u^\ve(t)\,  \partial_x v^\ve(t), \partial_t \partial_x  u^\ve(t) \rangle_Q \; . 
\end{eqnarray*}
Then, the term on the right-hand side can be rewritten  as 
\begin{eqnarray*}
\langle \chi^\ve(x)\,  u^\ve\, \partial_x v^\ve, \partial_t   \partial_x u^\ve \rangle_Q =
 \langle \chi^\ve(x)\,  \partial_t u^\ve\, \partial_x v^\ve + \chi^\ve(x)\, u^\ve \,\partial_t\partial_x v^\ve, \partial_x   u^\ve \rangle_Q\\
+ \, \partial_t \langle \chi^\ve(x) \,  u^\ve \partial_x v^\ve, \partial_x   u^\ve \rangle_Q\,  .
\end{eqnarray*}
The first and second  terms can be estimated as
\begin{eqnarray*}
 |\langle \chi^\ve(x) \partial_t u^\ve \partial_x v^\ve , \partial_x   u^\ve \rangle_Q|
 \leq \frac1 2 \| \partial_t u^\ve\|^2_{L^2(Q)}  +\frac{ (\chi^1)^2}2  \| \partial_x v^\ve \|^2_{L^\infty(Q)} \| \partial_x u^\ve \|^2_{L^2 (Q)} \; ,  
\end{eqnarray*}
and 
\begin{eqnarray*}
| \langle \chi^\ve(x)   u^\ve \partial_t\partial_x v^\ve,  \partial_x   u^\ve \rangle_Q|
 \leq  (\chi^1)^2 \|\partial_t\partial_x  v^\ve\|^2_{L^2(Q)}  +  \frac 1 4 \|u^\ve\|^2_{L^\infty(Q)} \| \partial_x   u^\ve\|^2_{L^2(Q)} \; . 
\end{eqnarray*}
Thus,  considering  the fact that $ \| \partial_x v^\ve \|_{L^\infty(Q_\tau)} \leq C $, we obtain 
\begin{eqnarray*}
\|\partial_t u^\ve\|^2_{L^2(Q)} + d_u^0 \, \partial_t \|\partial_x u^\ve\|^2_{L^2(Q)} \leq 
C_1 \big(\| \partial_x u^\ve \|^2_{L^2 (Q)} +  \|\partial_t \partial_x v^\ve\|^2_{L^2(Q)}\big)\\
 + C_2 \|u^\ve\|^2_{L^\infty(Q)} \| \partial_x   u^\ve\|^2_{L^2(Q)}
 +2 \partial_t \langle \chi^\ve(x) \, u^\ve \, \partial_x v^\ve, \partial_x   u^\ve \rangle_Q
\; . 
\end{eqnarray*}
For the last term we have that for $t \in (0,\tau]$
\begin{equation*}
\int_0^t \partial_s \langle \chi^\ve u^\ve \partial_x v^\ve, \partial_x   u^\ve \rangle_Q ds=
\langle \chi^\ve u^\ve(t) \partial_x v^\ve(t), \partial_x   u^\ve(t) \rangle_Q - \langle \chi^\ve u^\ve(0) \partial_x v^\ve(0), \partial_x   u^\ve(0) \rangle_Q
\end{equation*}
 and 
\begin{equation*}
\begin{aligned}
 |\langle \chi^\ve\, u^\ve(t) \partial_x v^\ve(t), \partial_x   u^\ve(t) \rangle_Q|+| \langle \chi^\ve\, u^\ve(0) \partial_x v^\ve(0), \partial_x   u^\ve(0) \rangle_Q|
 \;\leq\; \frac {d_u^0} 8 \|\partial_x   u^\ve(t) \|^2_{L^2(Q)} \\
 + \,
C_1  \|\partial_x v^\ve(t)\|^2_{L^\infty(Q)} \|u^\ve(t)\|^2_{L^2(Q)} + 
C_2\|\partial_x v_0\|^2_{L^\infty(Q)}\|u_0\|^2_{L^2(Q)} + C_3\|\partial_x u_0\|^2_{L^2(Q)}   .
\end{aligned}
\end{equation*}
Applying Gronwall's lemma and using estimates \eqref{estim_u_1}, \eqref{estim_u_2}, and \eqref{estim_v_2} along with  
$u_0 \in H^1(Q)$ and $v_0 \in H^2(Q)$, we obtain that for a.e.\ $t \in [0,\tau]$
\begin{eqnarray*}
\|\partial_x u^\ve(t)\|^2_{L^2(Q)} \leq 
C_1 \exp \big( \|u^\ve\|^2_{L^2(0,\tau; L^\infty(Q))} \big) + C_2  \leq C. 
\end{eqnarray*}
Thus, we conclude that 
\begin{eqnarray}\label{estim_u_3}
\|\partial_x u^\ve\|^2_{L^\infty(0,\tau; L^2(Q))}+  \|\partial_t u^\ve\|^2_{L^2(Q_\tau)}  \leq C . 
\end{eqnarray}

 To prove  uniqueness, we assume there are two solutions and  consider $u^\ve = u_1^\ve - u_2^\ve$ and
 $v^\ve= v^\ve_1 - v^\ve_2$  as test functions in equations   \eqref{micro_u} and \eqref{micro_v}, respectively,
 \begin{eqnarray*}
&&\langle u^\ve_t, u^\ve\rangle_{Q_\tau} 
+\langle D^\ve_u(x) \partial_x u^\ve,  \partial_x u^\ve \rangle_{Q_\tau}   -
\langle  \chi^\ve(x) (u^\ve \partial_x v^\ve_1+ u_2^\ve \partial_x v^\ve),  \partial_x u^\ve \rangle_{Q_\tau} 
= 0 ,   \\
&&\langle v_t^\ve,  v^\ve\rangle_{Q_\tau} 
+\langle D_v(x)\partial_x v^\ve , \partial_x v^\ve \rangle_{Q_\tau}  + \gamma \langle v^\ve, v^\ve\rangle_{Q_\tau} = \alpha \langle u^\ve, v^\ve \rangle_{Q_\tau}.
\end{eqnarray*}
Then using the  boundedness of $u^\ve_i$ and $\partial_x v^\ve_i$, $i=1,2$, along with Young's and Gronwall's inequalities,
we obtain 
$u_1^\ve= u^\ve_2$ and $v_1^\ve = v_2^\ve$  for a.e.\ $(t,x) \in Q_\tau$ and  $P$-a.s.\ $\omega\in \Omega$.
\end{proof}

\noindent {\bf Remark.}  The constant  $C$ in estimates  \eqref{apriori_estim} depends on $\tau$, i.e.\ $C \sim a e^{b\tau}$, $a,b>0$. 
 However, if $\text{dim}(Q)=1$ the solutions  of \eqref{micro_model} exist for any fixed $\tau >0$ without  any smallness restrictions on $u_0$ and $v_0$.  Moreover, the estimates \eqref{apriori_estim} are uniform in $\ve$.  \\

In the system investigated in this paper, the diffusion  $D^\ve_u$ and chemotaxis  $\chi^\ve$ coefficients depend on a small parameter $\ve$,  and we do not have estimates which are uniform in $\varepsilon$ for  $\nabla D^\ve_u$ and $\nabla \chi^\ve$. Hence, when  $\text{dim} (Q) =2$ we cannot use the derivation of the \text{a priori} estimates and the corresponding proof of well-posedness developed in  \cite{Nagai1997}. Instead, when $\text{dim} (Q) = 2$ or $\text{dim} (Q) = 3$  we adopt an approach similar to the one in \cite{CorriasPerthame2006}. 

\begin{theorem}\label{thm2_2}
Under Assumption~\ref{assumptions} and assuming  $d=\text{dim} (Q) = 2$ or $3$,  there exists a unique weak  solution of  \eqref{micro_model} for every $\ve >0$, and   we have
\begin{eqnarray}\label{apriori_estim_2}
\begin{aligned}
\|u^\ve\|_{L^\infty(0,\tau; L^2(Q))}  +
 \|\nabla u^\ve\|_{L^\infty(0,\tau; L^2(Q))}+  \|\partial_t u^\ve\|_{L^2(Q_\tau)} & \leq C, \\
  \|v^\ve \|_{L^\infty(0,\tau; H^1(Q))}+  \| \partial_t  v^\ve\|_{L^2(0,\tau; H^1(Q))} +
\| v^\ve\|_{L^2(0,\tau; H^2(Q))} & \leq C
\end{aligned}
\end{eqnarray}
for $P$-a.s.\  $\omega\in\Omega$ and a constant $C$ which is  independent of $\ve$.
\end{theorem} 
\begin{proof}
Similarly to Theorem~\ref{thm2} we obtain the non-negativity and the estimates  \eqref{L1_estim}  and \eqref{L1_estim2}  for the $L^1$-norms of $u^\ve$ and $v^\ve$.

Using the estimates for the derivatives of the Green function of the operator  $A=- \nabla\cdot (D_v(x)\nabla)$ (see, e.g.,\ \cite{CorriasPerthame2006, Mora, Winkler2010})  we obtain 
 $$
\| \nabla e^{-t (A+\gamma)} \phi \|_{L^{r_1}(Q)} \leq  C_1 \,  t^{-\frac 12 - \frac d2( \frac 1 {r_2} - \frac 1 {r_1}) } \,   \| \phi \|_{L^{r_2}(Q)}, \quad t>0, 
$$
for all $1\leq r_2 \leq r_1 \leq \infty$ and  $\phi \in L^{r_2}(Q)$, and 
 $$
\| \nabla e^{-t (A+\gamma)} \phi \|_{L^p(Q)} \leq  C_2 \| \nabla \phi \|_{L^p(Q)}, 
$$
for $2\leq p \leq \infty$, $\phi \in W^{1,p}(Q)$, and some constants $C_1$ and $C_2$ that depend on $Q$. Here, $\gamma$ is the decay constant in the second equation in \eqref{micro_model}.  
Applying the variation-of-constants formula, see e.g.~\cite{Pazy}, yields
$$
v^\ve(t, \cdot) = e^{-t(A+\gamma)} v_0(\cdot) + \alpha \int_0^t e^{(s-t)(A+ \gamma)} u^\ve(s, \cdot) ds.
$$
Then, for  $r_1$ and $r_2$ such that   $\dfrac 1 2 + \dfrac{d}2\left(\dfrac 1 {r_2} - \dfrac 1 {r_1}\right) < 1$, we have
\begin{equation}\label{estim_grad_v}
\|\nabla v^\ve(\cdot, t)\|_{L^{r_1}(\Omega)}  \leq C_v \Big(\|\nabla v_0 \|_{L^{r_1}(Q)} + \sup\limits_{s\in (0,t)} \| u^\ve(\cdot, s) \|_{L^{r_2}(Q)} \Big)  \quad \text{ for all } t\in (0,\tau].
\end{equation}
We now consider $|u^{\ve}|^{p-1}$, for some $p>1$, as a test function in \eqref{micro_u} to obtain 
\begin{equation}\label{estim_dd}
\begin{aligned}
\frac{d}{dt} \int_Q |u^\ve|^p dx + 4\frac{p-1} p  d_u^0\int_Q \big| \nabla |u^\ve|^{\frac p2} \big|^2 dx   \leq  
2(p-1) \chi^1 \int_Q   |u^\ve|^{\frac p 2}\,   \big|\nabla |u^\ve|^{\frac p 2}\big| \,  |\nabla v^\ve| dx.
\end{aligned}
\end{equation}
The integral on the right-hand side can be rewritten  as 
\begin{equation*}
I=\int_Q |u^\ve|^{\frac p2} \,  \big|\nabla |u^\ve|^{\frac p 2}\big| \,  |\nabla v^\ve| dx  \leq 
\|  \nabla |u^\ve|^{\frac p2}  \|_{L^2(Q)} \| |u^\ve|^{\frac p2}\|_{L^{q_1}(Q)} \|\nabla v^\ve \|_{L^{q_2}(Q)},
\end{equation*}
where $ 1/{q_1} + 1 /{q_2} =  1/2$. 

For $d = 2$ and any  $\zeta >0$, we consider $q_2=2 + \zeta$ and $q_1= 2+\frac 4 \zeta$.  Then, applying the Sobolev embedding  and estimate \eqref{estim_grad_v} with $1 + \frac \zeta{4+\zeta}<r_2\leq 2$, we obtain 

\begin{equation*}
\begin{aligned}
I & \leq 
 \|\nabla |u^\ve|^{\frac p2}\|_{L^2(Q)} \||u^\ve|^{\frac p 2}\|_{L^{2+  \frac 4  \varsigma}(Q)} \|\nabla v^\ve \|_{L^{2 +\zeta}(Q)}
\leq   \|  \nabla |u^\ve|^{\frac p2}  \|_{L^2(Q)}  \\
& \times C_b \Big(\|  \nabla |u^\ve|^{\frac p2} \|_{L^2(Q)} +  \||u^\ve|^{\frac p 2}\|_{L^1(Q)}\Big)  C_v \Big(\|\nabla v_0 \|_{L^{2+\zeta}(Q)} 
+ \sup\limits_{s\in (0,t)} \| u^\ve(s) \|_{L^{r_2}(Q)} \Big), 
\end{aligned}
\end{equation*}
where $C_b$ is the embedding constant. If $\| \nabla |u^\ve(t)|^{\frac p 2}  \|^2_{L^2(Q)} \geq 1$ for  $t \in (0, \tau]$ and $p=r_2$,  using the estimate for $I$  and  inequality  \eqref{estim_dd}  we obtain
\begin{equation}\label{estim_d2_emb}
\begin{aligned}
\frac{d}{dt} \int_Q |u^\ve|^{r_2} dx \leq 2(r_2-1) \| \nabla |u^\ve|^{\frac {r_2} 2}  \|^2_{L^2(Q)} \Big[C_bC_v\chi^1(1+ \|u_0\|^{r_2/2}_{L^1(Q)}|Q|^{\frac{2-r_2}2})  \\
\times \Big(\|\nabla v_0 \|_{L^{2+\zeta}(Q)} + \sup\limits_{s\in (0,t)} \| u^\ve(s) \|_{L^{r_2}(Q)}\Big)  - \frac {2d^0_u} {r_2} \Big].
\end{aligned}
\end{equation}  
If  for some $t \in (0, \tau]$ we have that $\| \nabla |u^\ve(t)|^{\frac {r_2} 2}  \|_{L^2(Q)} \leq  1$, then  using  the  Gagliardo-Nirenberg inequality \eqref{GN_In} with $s=2$, $\sigma=1/2$, $d=2$, $l=2$, and $q=1$ we obtain that 
\begin{equation}\label{estim_d2_small}
\begin{aligned}
\|u^\ve(t) \|^{r_2}_{L^{r_2}(Q)} \leq \tilde C \big(\| \nabla |u^\ve(t)|^{\frac {r_2} 2}  \|_{L^2(Q)}\||u^\ve|^{\frac {r_2} 2 }\|_{L^1(Q)} + \||u^\ve|^{\frac {r_2} 2}\|^2_{L^1(Q)}\big)
\\ \leq C_g\big(\|u_0\|^{\frac {r_2} 2}_{L^1(Q)} + \|u_0\|^{r_2}_{L^1(Q)}\big).
\end{aligned}
\end{equation}
For $d = 3$ we consider $q_2=3$, $q_1=6$,  and we apply  the Sobolev embedding theorem to obtain 
\begin{equation}\label{estim_d3_emb}
\begin{aligned}
I  \leq
 \|  \nabla |u^\ve|^{\frac p 2}\|_{L^2(Q)}\, C_b\Big(\|\nabla |u^\ve|^{\frac p 2}\|_{L^2(Q)} + \||u^\ve|^{\frac p 2}\|_{L^1(Q)} \Big)  \\
\times  C_v\Big( \|\nabla v_0 \|_{L^{3}(Q)}  +  \sup\limits_{s\in (0,t)} \| u^\ve(s) \|_{L^{r_2}(Q)} \Big), 
\end{aligned}
\end{equation}
where $C_b$ is the embedding constant and $3/2< r_2 \leq 2$.  If $\| \nabla |u^\ve(t)|^{\frac p 2}  \|_{L^2(Q)} \geq  1$ for $t \in (0, \tau]$ and  $p=r_2$, we have
\begin{eqnarray*}
&&\frac{d}{dt} \int_Q |u^\ve|^{r_2} dx \leq 2(r_2-1) \| \nabla |u^\ve|^{\frac {r_2} 2}  \|^2_{L^2(Q)}  \\
&& \times  \Big[C_b C_v \chi^1 \big(1+ \|u_0\|^{r_2/ 2}_{L^1(Q)}|Q|^{\frac{2-r_2}2}\big)\Big(\|\nabla v_0 \|_{L^{3}(Q)} + \sup\limits_{s\in (0,t)} \| u^\ve(s) \|_{L^{r_2}(Q)} \Big)  - \frac {2d_u^0}{r_2}\Big].
\end{eqnarray*}
If  for some $t \in (0, \tau]$ we have $\| \nabla |u^\ve(t)|^{\frac {r_2} 2}  \|_{L^2(Q)} \leq  1$, then  using  the Gagliardo-Nirenberg inequality \eqref{GN_In} with $s=2$, $\sigma=3/5$, $d=3$, $l=2$, and $q=1$ we obtain that 
\begin{equation}\label{estim_d3_GN}
\begin{aligned}
\|u^\ve(t) \|^{r_2}_{L^{r_2}(Q)} \leq \tilde C \big(\| \nabla |u^\ve(t)|^{\frac {r_2} 2}  \|^{\frac 6 5}_{L^2(Q)}\||u^\ve|^{\frac {r_2} 2 }\|^{\frac 4 5}_{L^1(Q)} + \||u^\ve|^{\frac {r_2} 2}\|^2_{L^1(Q)}\big)
\\ \leq C_g\big(\|u_0\|^{\frac {2r_2} 5}_{L^1(Q)} + \|u_0\|^{r_2}_{L^1(Q)}\big).
\end{aligned}
\end{equation}
Thus,  if $\big(1+ \|u_0\|^{\frac{r_2} 2}_{L^1(Q)}|Q|^{\frac{2-r_2}2}\big)\big(\|\nabla v_0 \|_{L^{q}(Q)} + \sup_{s\in (0,t)} \| u^\ve(\cdot, s) \|_{L^{r_2}(Q)}\big)$  is sufficiently small  we obtain that  $\|u^\ve(t)\|^{r_2}_{L^{r_2}(Q)} $ is monotone decreasing  for all $t \in (0, \tau]$ such that $\| \nabla |u^\ve(t)|^{\frac {r_2} 2}  \|_{L^2(Q)} \geq  1$.  Here, $q=\max\{2+\zeta, d\}$ for any $\zeta>0$.  For any $t\in (0, \tau]$ such  that $\| \nabla |u^\ve(t)|^{\frac {r_2} 2}  \|_{L^2(Q)} \leq  1$ we have 
$$\|u^\ve(t) \|_{L^{r_2}(Q)} \leq C_g \big(\|u_0\|^{\frac 2 {d+2}}_{L^1(Q)} + \|u_0\|_{L^1(Q)}\big) .$$
Hence,   if  $v_0$ and $u_0$ satisfy assumption \eqref{estim_assum_in}, then
\begin{equation}
 \|u^\ve\|_{L^\infty(0,\tau; L^{r_2}(Q))} \leq \max\big\{\|u_0 \|_{L^{r_2}(Q)},  C_g\big( \|u_0 \|_{L^1(Q)} + \|u_0 \|^{\frac 2 {d+2}}_{L^1(Q)}\big)\big\} .
 \end{equation}
 
 Using the last estimate together with estimate \eqref{estim_grad_v} and taking $u^\ve$ as a test function in \eqref{micro_u} we have
 $$
 \|u^\ve\|_{L^\infty(0,\tau; L^2(Q))}  + \|\nabla u^\ve\|_{L^2(Q_\tau)}  \leq C,
 $$
 with a constant $C$ independent of $\ve$. Considering $ v^\ve$ and $\partial_t v^\ve$ as  test functions in \eqref{micro_v}  we obtain  
\begin{equation*}
 \| v^\ve\|_{L^\infty(0,\tau; L^2(Q))}+ \| \partial_t v^\ve\|_{L^2(Q_\tau)} + \|\nabla v^\ve\|_{L^\infty(0,\tau; L^2(Q))} \leq C_1 (\|u^\ve\|_{L^2(Q_\tau)} + \|v_0\|_{H^1(Q)})\leq C.
\end{equation*}
 Taking $|u^\ve|^{p-1}$ as a test function in \eqref{micro_u} with $p>d$ yields 
 \begin{equation}\label{up_estim}
 \|u^\ve\|_{L^\infty(0,\tau; L^p(Q))} \leq C. 
 \end{equation}
Thus, applying  \eqref{estim_grad_v} and using the estimate \eqref{up_estim} with $p>d$, we obtain
 $$
 \|\nabla  v^\ve \|_{L^\infty(Q_\tau)} \leq C_1   \|u^\ve\|_{L^\infty(0,\tau; L^p(Q))}  \leq C_2.
 $$
 Then, considering $\partial_t u^\ve$ as a test function in \eqref{micro_u} ensure 
 $$
 \|\partial_t u^\ve\|_{L^2(Q_\tau)} + \| \nabla u^\ve \|_{L^\infty(0,\tau, L^2(Q))} \leq C.  
 $$
 Taking $\Delta v^\ve$  and $\Delta \partial_t v^\ve$  as  test functions in  \eqref{micro_v} and applying zero Neumann boundary conditions for $u^\ve$ result in
 $$
 \| \partial_t \nabla v^\ve \|_{L^2(Q_\tau)} + \| \nabla^2 v^\ve \|_{L^\infty(0,\tau; L^2(Q))} \leq C. 
 $$
 
 As in Theorem~\ref{thm2} we  obtain the existence of a weak solution of \eqref{micro_model} in $Q_\tau$ by applying the Galerkin method and a fixed point argument.
 Similarly, we show the uniqueness of the weak solution of \eqref{micro_model} by considering the equations for the difference of two solutions and showing that  they are equal a.e.\ in $Q_\tau$ and $P$-a.s. in $\Omega$.
 \end{proof}

\section{Stochastic homogenization} \label{stoch_homog} 

In this section, we derive our main homogenization result for problem \eqref{micro_model}. The system of  macroscopic equations is obtained in Theorem \ref{main_thm} by using  the concept of stochastic two-scale convergence introduced in \cite{Zhikov_Piatnitsky_2006}.  For the reader's convenience we state the general definition of  two-scale convergence by means of Palm measures, and then apply it to the specific context of the problem studied in this paper. In the following, we also make use of the notions of invariance and ergodicity, which we now define. 

\begin{definition}
A measurable function $f$ on $\Omega$ is said to be {\it invariant} for a dynamical system $\T(x)$  if for each $x\in\mathbb R^d$, $f(\omega)=f(\T(x)\omega)$, $P$-a.s.\ on $\Omega$.
\end{definition}
\begin{definition}
A dynamical system $\T(x)$ is said to be {\it ergodic}, if every measurable function which is invariant for $\T(x)$ is $P$-a.s.\ equal to a constant.
\end{definition}

The random environment described by the coefficients in \eqref{micro_model} can also be characterized in terms of a random measure, which is defined as follows.

\begin{definition} 
Let  $(\Omega, \mathcal{F})$ be a measurable space and $\mathcal{B}(\mathbb{R}^d)$ be the $\sigma$-algebra of Borel sets in $\mathbb{R}^d$.  A mapping $\tilde\mu: \Omega\times\mathcal{B}(\mathbb{R}^d)\rightarrow\mathbb{R}_{+}\cup\{\infty\}$ is called a random measure on $(\mathbb{R}^d,\mathcal{B}(\mathbb{R}^d))$ if the function $\mu_\omega(A)=\tilde\mu(\omega,A)$  is $\mathcal{F}$-measurable in $\omega\in\Omega$ for each $A\in\mathcal{B}(\mathbb{R}^d)$ and a measure in $A\in\mathcal{B}(\mathbb{R}^d)$ for each $\omega\in\Omega$.
\end{definition}

Even though more general definitions of a random measure exist in the literature (see, e.g., \cite{Daley2008} or \cite{Kal}), in the remainder of the paper $\mu_\omega$ will always denote a random measure on $(\mathbb{R}^d,\mathcal{B}(\mathbb{R}^d))$. 

\begin{definition}
The {\it Palm measure} of the random measure $\mu_\omega$ is the measure $\boldsymbol\mu$ on $(\Omega, \mathcal F)$ defined by the relation
\begin{equation}
\boldsymbol\mu(A)= \int_\Omega \int_{\mathbb R^d} \mathbb I_{[0,1)^d}(x) \,  \mathbb I_A(\T(x)\omega) \, d \mu_\omega(x) dP(\omega) \; , 
\end{equation}
where $\mathbb I_{K}$ denotes the characteristic function of the set $K$.
\end{definition}

The value of the notion of a Palm measure is that it allows for a generalization of Birkhoff's ergodic theorem for stationary random measures. Specifically, given a dynamical system $\T(x)$, we say that the random measure $\mu_\omega$ is stationary  if for every $\phi\in C_0^\infty(\mathbb{R}^d)$ 
$$\int_{\mathbb{R}^d}\phi(y-x)\,d\mu_\omega(y) \,=\, \int_{\mathbb{R}^d} \phi(y)\,d\mu_{\T(x)\omega}(y) \; .$$
The intensity $m(\mu_\omega)$ of a random measure $\mu_\omega$ is defined by
\begin{equation}\label{def_intensity}
m(\mu_\omega)= \int_\Omega\int_{[0,1)^d} \,  d \mu_\omega(x) \, dP(\omega) \; .
\end{equation}

\begin{theorem}[Ergodic theorem \cite{Zhikov_Piatnitsky_2006}]\label{ergodic_thm}
Let the dynamical system $\T(x)$ be ergodic and assume that the stationary random measure $\mu_\omega$ has finite intensity $m(\mu_\omega)>0$. Then 
\begin{equation}\label{ergodic_equality}
\lim\limits_{t\to \infty} \frac 1 { t\,|A|} \int_{ tA} g( \T(x) \omega) d \mu_\omega(x) = \int_\Omega g(\omega) d \boldsymbol\mu(\omega) \quad \text{ a.s. with respect to } P 
\end{equation} 
for all bounded Borel sets $A$, with volume $|A|>0$, and all $g \in L^1(\Omega, \boldsymbol\mu)$. 
\end{theorem}

We remark that  for  $\boldsymbol\mu =P$ (i.e.,\ $d\mu_\omega(x) = dx$),  Theorem~\ref{ergodic_thm} reduces to the classical ergodic theorem of Birkhoff.

We now define the notion of stochastic two-scale convergence, which is one of the
main tools used in proving Theorem \ref{main_thm}. We consider the family of random measures 
$$
d\mu_\omega^\ve(x) = \ve^d d\mu_\omega\left(\frac x \ve \right) \; . 
$$

\noindent We remark that an immediate consequence of Theorem \ref{ergodic_thm} is that on every compact subset of $\mathbb{R}^d$, the family $d\mu^\ve_\omega(x)$ converges weakly to the  deterministic measure  $m(\mu_\omega)\, dx$ a.s. with respect to $P$ as $\ve \to 0$ (see, e.g.,  \cite{Zhikov_Piatnitsky_2006}). 

\begin{definition}[Stochastic two-scale convergence \cite{Heida_2011, Zhikov_Piatnitsky_2006}]\label{def_t_s}
 Let $Q$ be a domain in $\mathbb{R}^d$,  $\T(x)$ be an  ergodic dynamical system, and $\T(x)\tilde \omega$ be a ``typical trajectory,'' i.e. one that satisfies equation \eqref{ergodic_equality} for all $g\in C(\Omega)$.  Then, we say that a sequence $\{ v^\ve \} \subset L^2(0,\tau; L^2(Q, \mu_{\tilde \omega}^\ve))$ 
converges  stochastically  two-scale to   $v\in L^2(0,\tau; L^2(Q\times \Omega, dx \times d\boldsymbol\mu(\omega)))$ if 
\begin{equation}\label{estim_t-s-def}
\limsup\limits_{\ve \to 0} \int_0^\tau\int_Q |v^\ve(t, x)|^2 \,  d \mu_{\tilde \omega}^\ve(x) \,  dt < \infty
\end{equation}
and
\begin{eqnarray}\label{stoch_two_scale}
&&\lim\limits_{\ve \to 0} \int_0^\tau\int_{Q} v^\ve(t, x) \varphi(t,x) b(\T(x/\ve)\tilde \omega)\, d\mu^\ve_{\tilde \omega}(x) dt \\ && \hspace{ 4 cm } = 
\int_0^\tau\int_Q\int_\Omega v(t, x, \omega) \varphi(t,x) b(\omega)  \, d\boldsymbol\mu(\omega)  dx  dt \nonumber
\end{eqnarray}
for all   $\varphi \in C^\infty_0([0, \tau)\times Q)$ and $b \in L^2(\Omega, \boldsymbol\mu)$. 
\end{definition}

 It is evident that if  $Q\subset\mathbb{R}^d$ is bounded, each $\varphi \in C^\infty(\overline Q_\tau)$ can be used as a test function in Definition~\ref{def_t_s}. 
The concept of  a ``typical trajectory'' in Definition \ref{def_t_s} extends to realizations $\tilde \omega\in\Omega$. Specifically, we say that $\tilde\omega \in \Omega$ is a ``typical realization'' if  \eqref{ergodic_equality} holds true at $\tilde\omega$ for all $g\in C(\Omega)$. 

\begin{theorem}\label{c0} \textnormal{\cite{Heida_2011, Zhikov_Piatnitsky_2006}}
Every sequence $\{v^\ve \}\subset L^2(0,\tau; L^2(Q, \mu_{\tilde \omega}^\ve)) $ that satisfies \eqref{estim_t-s-def}  converges  along a subsequence to 
some $v\in L^2(0,\tau; L^2(Q\times \Omega, dx \times d\boldsymbol\mu(\omega)))$ in the sense of  stochastic two-scale convergence.
\end{theorem}

Before we proceed, we need to define a concept of stochastic derivative and the space $H^1(\Omega,\boldsymbol\mu)$ for the Palm measure $\boldsymbol\mu$. First, we say that a function $u\in C(\Omega)$ belongs to $C^1(\Omega)$ if the limit 
$$\partial^j_\omega u(\omega)\,=\,\lim_{h\rightarrow 0 }\frac{u(\T(he_j)\omega)-u(\omega)}{h}$$
exists and $\partial^j_\omega u(\omega)\in C(\Omega)$. Then, the Sobolev space $H^1(\Omega,\boldsymbol\mu)$ is defined as follows.

\begin{definition} \textnormal{\cite{Zhikov_Piatnitsky_2006}} We say that a function $u\in L^2(\Omega,\boldsymbol\mu)$ belongs to $H^1(\Omega,\boldsymbol\mu)$ and $\partial_\omega u$ is a (stochastic) derivative of $u$  if there exists a sequence $u_k\in C^1(\Omega)$ such that $u_k\rightarrow u$ in $L^2(\Omega,\boldsymbol\mu)$ and $\partial^j_\omega u_k\rightarrow\partial^j_\omega u$ in $L^2(\Omega,\boldsymbol\mu)$.
\end{definition}

 In general, the stochastic derivative $\partial_\omega u$ does not have to be unique (see \cite{Zhikov_Piatnitsky_2006} for counterexamples). We remark, however, that the particular setting of our problem yields the uniqueness of $\partial_\omega u$. We also define  $L^2_{\text{pot}}(\Omega,\boldsymbol\mu)$ and   $L^2_{\text{sol}}(\Omega,\boldsymbol\mu)$ to be the spaces of  potential functions and  divergence-free functions, respectively. More precisely, 
 $$L^2_{\text{pot}} (\Omega,\boldsymbol\mu)= \overline{ \{\partial_\omega u \,: \, u \in C^1(\Omega) \} }\,\mbox{ and }\,  L^2_{\text{sol}} (\Omega,\boldsymbol\mu)= \big(L^2_{\text{pot}} (\Omega,\boldsymbol\mu)\big)^{\perp},$$
where the closure in the definition of $L^2_{\text{pot}} (\Omega,\boldsymbol\mu)$ is with respect to the  $L^2(\Omega,\boldsymbol\mu)$ norm.

We now state two compactness results for the notion of stochastic two-scale convergence to be used in the following. Theorems \ref{c1}  and \ref{c2} were proved in \cite{Zhikov_Piatnitsky_2006} in the more general setting  of an arbitrary random measure. Here, the theorems are stated in the context of our problem, i.e. for  a non-degenerate random measure $\mu_\omega$ (see \cite{Zhikov_Piatnitsky_2006} for the definition of non-degeneracy). 

 \noindent{\bf Remark.} For a non-degenerate measure,  $\partial_\omega^j$  denotes  the  generator of a strongly continuous group of unitary operators in $L^2(\Omega, \boldsymbol\mu)$ associated with $\T(x)$ along  the $e_j$ direction. The domains of $\partial_\omega^j$, with $j=1,\ldots, d$,  are dense in $L^2(\Omega, \boldsymbol\mu)$.  We let $\nabla_\omega u = (\partial_\omega^1 u, \ldots,  \partial_\omega^d u)^T$ and $H^1(\Omega, \boldsymbol\mu) = \{ v\in L^2(\Omega, \boldsymbol\mu) \, : \,  \nabla_\omega v \in L^2(\Omega, \boldsymbol\mu) \}$.\\

\begin{theorem}\label{c1} \textnormal{\cite{Zhikov_Piatnitsky_2006}}
Let $Q$ be a domain in $\mathbb{R}^d$  and assume that  $\mu_\omega$ is a non-degenerate random measure and that the  sequence $\{v^\ve\} \subset H^1(Q, \mu_{\tilde \omega}^\ve)$ is such that 
$$ \| v^\ve \|_{L^2(Q, \mu_{\tilde \omega}^\ve)} \leq C(\tilde \omega) \; , \qquad \| \nabla v^\ve \|_{L^2(Q, \mu_{\tilde \omega}^\ve)} \leq C(\tilde \omega) \; . 
$$
Then there exist functions $v\in H^1(Q)$ and $v_1 \in L^2(Q; L^2_{\text{pot}}(\Omega, \boldsymbol\mu))$ such that, up to a  subsequence, the following hold:
\begin{equation}
\begin{aligned}
v^\ve &\rightharpoonup  v \, \quad && \text{ stochastically two-scale}  , \\
\nabla v^\ve & \rightharpoonup   \nabla_x v + v_1   \qquad && \text{ stochastically two-scale} . 
\end{aligned}
\end{equation}
\end{theorem} 

\begin{theorem}\label{c2} \textnormal{\cite{Zhikov_Piatnitsky_2006}}
Let $Q$ be a domain in $\mathbb{R}^d$ and assume that  $\mu_\omega$ is a non-degenerate random measure and that  the  sequence $\{v^\ve\} \subset H^1(Q, \mu_{\tilde \omega}^\ve)$ is such that 
$$ \| v^\ve \|_{L^2(Q, \mu_{\tilde \omega}^\ve)} \leq C(\tilde \omega)\; , \qquad \ve \| \nabla v^\ve \|_{L^2(Q, \mu_{\tilde \omega}^\ve)} \leq C(\tilde \omega) \; . 
$$
Then there exists a function $v\in L^2(Q; H^1(\Omega, \boldsymbol\mu))$  
such that, up to a  subsequence, the following hold:
\begin{equation}
\begin{aligned}
v^\ve & \rightharpoonup  v  && \text{ stochastically two-scale}, \\
\ve \nabla v^\ve & \rightharpoonup  \nabla_\omega v \qquad && \text{ stochastically two-scale} .
\end{aligned} 
\end{equation}
\end{theorem} 

\noindent Similar results hold for $\{v^\ve \}\subset L^2(0,\tau; H^1(Q, \mu_{\tilde \omega}^\ve))$, where the time variable is considered as a parameter~\cite{Heida_2011}. 

In the following theorems, the Palm measure reduces to the probability measure $P$, i.e.,  $\boldsymbol{\mu}=P$. We now state and prove  the main homogenization result of this paper. 

\begin{theorem}\label{main_thm}
We assume  that the dynamical system $\T(x)$ is ergodic  and that the coefficients $D_u^\ve$,  $\chi^\ve$, and $D_v$ along with the initial conditions $u_0$ and $v_0$ satisfy Assumption~\ref{assumptions}.  
Then,  the sequence of weak solutions $\{ u^\ve, v^\ve\}$ of  the microscopic problem \eqref{micro_model} converges strongly in $L^2(Q_\tau)$ and weakly in $L^2(0,\tau; H^1(Q))$ to the solution  $(u, v) \in L^2(0,\tau; H^1(Q))^2$ of the macroscopic model:
\begin{eqnarray}\label{limit_eq}
\begin{aligned}
&\partial_t u = \nabla\cdot ( D^\ast \nabla u- \chi^\ast \, u \, \nabla v)  \;  \quad && \text{in }  Q_\tau,   \\
&\partial_t v = \nabla\cdot (D_v(x) \nabla v) - \gamma v + \alpha u\; \quad&& \text{in } Q_\tau,\\
&(D^\ast \nabla u- \chi^\ast \, u \, \nabla v)\cdot n =0, \quad \nabla v\cdot n  =0 \; \quad&& \text{on } (0,\tau)\times \partial Q, \\
&u(0,x)= u_0(x), \quad v(0,x)=v_0(x) \qquad && \text{in } Q, 
\end{aligned}
\end{eqnarray}
$P$-a.s. in $\Omega$. The effective (macroscopic)  diffusion and chemotaxis matrices are defined as 
\begin{eqnarray}\label{effect_coef}
\, \,\qquad  D^\ast \xi  = \int_\Omega \widetilde D_u(\omega) (\bar u_{1,\xi}+\xi) \, d P(\omega) , \; \;   \chi^\ast \xi  = - \int_\Omega \big(\widetilde D_u(\omega)
\hat u_{1,\xi} - \widetilde \chi(\omega)\xi\big) \, d P(\omega)   
\end{eqnarray}
for any $\xi \in \mathbb R^d$, where   
 $\bar u_{1,\xi}, \, \hat u_{1, \xi}$  are solutions of the auxiliary problems
  \begin{eqnarray}
&&\bar u_{1, \xi} \in L^2_{\text{pot}}(\Omega) \quad \text{ such that } \quad  \widetilde D_u(\omega)( \bar u_{1, \xi} + \xi) \in L^2_{\text{sol}}(\Omega) \;, \label{unit_D}\\
&&  \hat u_{1, \xi} \in L^2_{\text{pot}}(\Omega) \quad \text{ such that } \quad  \widetilde D_u(\omega)\hat u_{1, \xi} - \widetilde \chi(\omega)\xi  \in L^2_{\text{sol}}(\Omega) \; . \label{unit_chi}
 \end{eqnarray}
 \end{theorem} 
 
\begin{proof}
From the \textit{a priori} estimates in \eqref{apriori_estim}, we obtain that 
\begin{eqnarray*}
u^\ve, \, \nabla u^\ve, \,  \partial_t u^\ve, \, v^\ve, \,  \nabla v^\ve,\, \nabla^2 v^\ve,\, \partial_t v^\ve, \, \partial_t \nabla v^\ve
\end{eqnarray*}
are bounded sequences in $L^2(Q_\tau)$ for  $P$-a.s.\ $\omega \in \Omega.$ Then, using Theorem \ref{c1} with $\boldsymbol\mu = P$, we obtain that, up to a subsequence, 
\begin{eqnarray*}
&u^\ve \,  { \rightharpoonup } \,   u   &\text{stochastically  two-scale,}  \quad  u \in L^2(0,\tau; H^1(Q)), \\
&\nabla  u^\ve \,  { \rightharpoonup } \,  \,  \nabla u+ u_1  \qquad &\text{stochastically two-scale,} \quad  u_1 \in L^2(Q_\tau; L^2_{pot}(\Omega)),  \\
&\partial_t u^\ve \,  { \rightharpoonup } \,  \tilde u &\text{stochastically two-scale,}  \quad  \tilde u \in L^2(Q_\tau; L^2( \Omega)), \\
& v^\ve \,  { \rightharpoonup }  \, v  & \text{stochastically two-scale}, \quad  v \in  L^2(0,\tau; H^1(Q)), \\
&\partial_t v^\ve \,  { \rightharpoonup } \,  \tilde v &\text{stochastically two-scale,}  \quad  \tilde v \in L^2(0, \tau; H^1(Q)), \\
&\nabla v^\ve \, { \rightharpoonup } \,  \hat v  & \text{stochastically two-scale}, \quad  \hat v \in L^2(0,\tau; H^1(Q))
\end{eqnarray*}
for all ``typical'' realizations  $\omega$. 

Now, considering the stochastic two-scale convergence of $u^\ve$ and $\partial_t u^\ve$,  we have that for $\varphi \in C^\infty_0(Q_\tau)$,  $b \in L^2(\Omega)$  and any ``typical'' realization $\tilde \omega \in \Omega$
\begin{eqnarray*}
&& \int_{Q_\tau} \int_{\Omega} \tilde u (t,x,\omega)  \varphi(t,x)  b(\omega) d P(\omega) dx dt = \lim\limits_{\ve \to 0} \int_{Q_\tau}
 \partial_t u^\ve(t,x)  \varphi(t,x)  b(\T(x/\ve) \tilde \omega)   dx dt =\\
 &&
- \lim\limits_{\ve \to 0} \int_{Q_\tau}  u^\ve(t,x)   \partial_t \varphi(t,x)  b(\T(x/\ve) \tilde \omega)  dx  dt
= -\int_{Q_\tau}\int_{\Omega}  u (t,x)  \partial_t \varphi(t,x)   b(\omega) \, d P(\omega) dx dt \\
&&=
\int_{Q_\tau}\int_{\Omega}  \partial_t u(t,x)  \varphi(t,x)   b(\omega)  dP(\omega) dx dt \; .
\end{eqnarray*}
Thus,  $ \tilde u(t,x,\omega) = \partial_t u(t,x)$ for  a.e.\ $(t,x)\in Q_\tau$ and  $P$-a.s.\  $\omega\in\Omega$. Similarly we conclude  that $\tilde v(t,x) = \partial_t v(t,x)$ for  a.e.\ $(t,x)\in Q_\tau$. 

From the definition of stochastic two-scale convergence of  $\nabla v^\ve$, we obtain that  for $\varphi \in C^\infty_0(Q_\tau)$,  $b \in L^2(\Omega)$  and any ``typical'' realization $\tilde \omega \in \Omega$
\begin{eqnarray*}
 \lim\limits_{\ve \to 0} \int_{Q_\tau}
 \nabla v^\ve(t,x) \,  \varphi(t,x)\,  b(\T(x/\ve)\tilde \omega) \,  dx dt=
 \int_{Q_\tau} \int_{\Omega} \hat v(t,x) \,   \varphi(t,x) b(\omega) \, d P(\omega) dx dt \, . 
 \end{eqnarray*}
 The weak convergence of $v^\ve$ in $L^2(0,\tau; H^1(Q))$, which is ensured by the  {\it a priori} estimates, implies that
 \begin{eqnarray*}
 \lim\limits_{\ve \to 0} \int_{Q_\tau}
 \nabla v^\ve(t,x) \,  \varphi(t,x)  \,  dx dt=
 \int_{Q_\tau} \nabla v(t,x) \,  \varphi(t,x)  \, dx dt 
 \end{eqnarray*}
for $P$-a.s.\ $\omega \in \Omega$ and  $\varphi \in L^2(Q_\tau)$. 
Thus, by choosing $b(\omega) =1$, we conclude that $\hat v(t,x) = \nabla v(t,x)$ for a.e.\ $(t,x) \in Q_\tau$.
Hence,  the stated {\it a priori} estimates and the Aubin-Lions compactness lemma \cite{Lions69} ensure that, up to a subsequence,   
$u^\ve \to u$,  $v^\ve \to v$,  and  $\nabla v^\ve \to \nabla v$ strongly in $L^2(Q_\tau)$ as $\ve\to 0$, $P$-a.s.

We now derive the macroscopic equations. 
Choosing $\psi \in C^\infty(\overline Q_\tau)$  as test function in \eqref{micro_v}, 
and by considering the  weak convergence of $u^\ve$ and $v^\ve$,  we obtain 
\begin{eqnarray*}
\langle v_t, \psi\rangle_{Q_\tau} 
+\langle D_v(x)\nabla  v , \nabla \psi \rangle_{Q_\tau}  + \gamma \langle v, \psi\rangle_{Q_\tau} = \alpha \langle u, \psi\rangle_{Q_\tau} \; . 
\end{eqnarray*}
Now, we consider  
$\phi(t,x) =\varphi(t,x)  +  \ve \varphi_1(t,x) \varphi_2(\T(x/\ve)\omega)$,  where $\varphi \in C^\infty(\overline Q_\tau)$, \,  $\varphi_1 \in C^\infty_0(Q_\tau)$ and $\varphi_2 \in  C^1(\Omega)$,  as test function  in \eqref{micro_u}  and obtain
 \begin{eqnarray}\label{homog_conver_u}
\quad  \begin{aligned}
\big\langle D^\ve_u \, \nabla u^\ve - \chi^\ve\,  u^\ve \nabla v^\ve , \nabla \varphi +  \ve \nabla  \varphi_1 \,  \varphi_2(\T(x/\ve)\omega) 
+  \varphi_1 \,   \nabla_\omega\varphi_2(\T( x/\ve)\omega)  \big \rangle_{Q_\tau} \\
+ \langle u^\ve_t, \varphi  +  \ve \varphi_1 \varphi_2(\T(x/\ve)\omega)\rangle_{Q_\tau}  = 0 \; .  
\end{aligned}
\end{eqnarray}
The stochastic two-scale limit in   \eqref{homog_conver_u} and  the strong convergence of $u^\ve$  yield as $\ve\to 0$
 \begin{eqnarray}\label{macro_1}
&& \quad \langle u_t, \varphi\rangle_{Q_\tau} 
+\langle \widetilde D_u(\omega)(\nabla u+ u_1) - \widetilde \chi(\omega) u \, \nabla v, \nabla  \varphi   +  \varphi_1 \, \nabla_\omega \varphi_2(\omega)\rangle_{Q_\tau,\Omega} = 0 \; . 
\end{eqnarray}
Choosing $\varphi(t,x)=0$ for $(t,x) \in Q_\tau$ we obtain 
 \begin{eqnarray*}
\langle \widetilde D_u(\omega)(\nabla u+ u_1) - \widetilde \chi(\omega)\,  u\,  \nabla v,  \varphi_1(t, x) \, \nabla_\omega \varphi_2(\omega)\rangle_{Q_\tau,\Omega} = 0 \; 
\end{eqnarray*}
for every $\varphi_1 \in C^\infty_0(Q_\tau)$ and $\varphi_2 \in  C^1(\Omega)$. Thus,  we have that for  $dt\times dx$-a.e.\ in $Q_\tau$ 
 \begin{eqnarray}\label{unit_cell_prob}
 \langle \widetilde D_u(\omega)(\nabla u+ u_1) -\widetilde \chi(\omega)\,  u\,  \nabla v,  \nabla_\omega \varphi_2 \rangle_\Omega  =0  \; . 
 \end{eqnarray} 
 Due to  the stated assumptions on $\widetilde D_u$ and $\widetilde \chi$ there exists a unique solution   $u_1(t,x, \cdot) \in L^2_{\text{pot}}(\Omega)$ of \eqref{unit_cell_prob}  that depends linearly on   $\nabla_x u(t,x)$ and  $u(t,x)\, \nabla_x v(t,x)$  for a.e.\ $(t,x) \in Q_\tau$, see e.g.\ \cite{ZKO1994}. 
We  consider  
$$u_1(t,x,\omega) = \sum_{j=1}^d \partial_{x_j} u(t,x) \,  \bar u^j_1(\omega)+  u(t,x)\,\sum_{j=1}^d  \partial_{x_j} v(t,x) \,  \hat u^j_1(\omega)
$$ for a.e.\  $(t,x) \in Q_\tau$ and $P$-a.s.\ $\omega\in\Omega$, and obtain from \eqref{unit_cell_prob} that  $\bar u^j_1, \, \hat u^j_1 \in L^2_{\text{pot}}(\Omega)$,  for $j=1,\ldots, d$,  are solutions of the   problems \eqref{unit_D} and \eqref{unit_chi}, respectively. Considering now $\varphi_1 =0$ in \eqref{macro_1},   and using the above expression for $u_1$, we obtain the macroscopic model \eqref{limit_eq}    with effective coefficients  $D^\ast$ and $\chi^\ast$  given by \eqref{effect_coef}.

By the stochastic two-scale convergence of $u^\ve$ and   $\partial_t u^\ve$, and the initial condition 
$u^\ve(0,x)= u_0(x)$, we obtain for  all $\varphi \in C^\infty_0([0, \tau)\times Q)$,  $b \in L^2(\Omega)$  and any ``typical'' realization $\tilde \omega \in \Omega$ that
\begin{eqnarray*}
&&\int_{Q_\tau} \int_{\Omega}  \partial_t u(t,x)  \varphi(t,x) b(\omega)  d P(\omega) dx dt =
\lim\limits_{\ve \to 0} \int_{Q_\tau}  \partial_t u^\ve(t,x)  \varphi(t,x) b(\T(x/\ve)\tilde\omega) dx dt \\
&& = 
- \lim\limits_{\ve \to 0} \int_{Q_\tau}  u^\ve(t,x)  \partial_t \varphi(t,x) b(\T(x/\ve)\tilde\omega) dx dt +
 \lim\limits_{\ve \to 0} \int_{Q}  u_0(x) \varphi(0,x) b(\T(x/\ve)\tilde\omega) dx dt\\
 &&=
- \int_{Q_\tau} \int_{\Omega}   u(t,x) \partial_t \varphi(t,x) b(\omega) d P(\omega) dx dt +
\int_{Q} \int_{\Omega}   u_0(x)  \varphi(0,x) b(\omega) d P(\omega) dx dt. 
\end{eqnarray*}
Similar calculations for $v^\ve$ ensure that the initial conditions $u(0,x)=u_0(x)$ and $v(0,x)=v_0(x)$ are satisfied a.e.\ in $Q$. 

The proof of the uniqueness of the solution is similar to the corresponding proof for the microscopic problem, and hence the convergence of the whole sequences $\{u^\ve\}$ and $\{v^\ve\}$ follows. Since \eqref{limit_eq} has a unique solution, and $D^\ast$ and $\chi^\ast$ do not depend on $\omega$, it follows that the solution of $\eqref{limit_eq}$ does not depend on $\omega$ either.
\end{proof}

\section{Periodic approximation of the effective coefficients}

We now turn our attention to the problem of approximating the homogenized coefficients shown in \eqref{effect_coef} by means of a periodization  procedure. The significance of such approximations is discussed in  \cite{BP2004} and \cite{Ow2003}. Here, we build upon the methods developed in \cite{BP2004} and consider the following periodization procedure.

We let $S_\rho=[0,\rho]^d$ for some $\rho>0$, and for each  $\omega\in\Omega$ we consider the periodic functions 
$$
D^\rho_{u,\text{per}}(z,\omega) = \widetilde D_u(\T({z(\text{mod} S_\rho)}) \omega), \quad 
\chi^\rho_{\text{per}}(z,\omega) = \widetilde \chi(\T({z(\text{mod} S_\rho)}) \omega) .
$$
Then for $P$-a.s.\ $\omega \in \Omega$, we consider the equations
\begin{equation*}
\begin{aligned}
\bar u^\ve_t &=\nabla \cdot  (D_{u,\text{per}}^\rho(x/\ve, \omega)\nabla \bar u^\ve-
\chi^\rho_{\text{per}}(x/\ve, \omega) \bar u^\ve\, \nabla \bar v^\ve)  && \text{ in } Q_\tau ,\\
\bar v^\ve_t &=\nabla\cdot (D_{v}(x)\nabla \bar v^\ve) - \gamma \bar v^\ve + \alpha \bar u^\ve  && \text{ in } Q_\tau, \\
& \nabla \bar u^\ve \cdot n =0, \quad    \nabla \bar v^\ve \cdot n =0 && \text{ on } (0, \tau)\times \partial Q.
\end{aligned}
\end{equation*}
 The equation for $\bar u^\ve$  has periodic coefficients,  and hence we can employ methods pertaining to periodic homogenization to obtain the  effective coefficients for the corresponding macroscopic problem. However, since $D^\rho_{u,\text{per}}(z,\omega) $ and $\chi^\rho_{\text{per}}(z,\omega) $ are not ergodic anymore, the effective coefficients are not deterministic (i.e., they depend on $\omega\in\Omega$). 

The unit cell problems that are obtained from the periodic homogenization approach are: 
Find   $\bar \eta^\rho_j, \, \hat \eta^\rho_j \in H^1_{\text{per}} (S_\rho)$, for $j=1,\ldots, d$,  such that 
\begin{eqnarray} \label{unit_per_u_v}
\begin{aligned}
\nabla_z \cdot (D_{u,\text{per}}^\rho(z, \omega)(\nabla_z \bar \eta^\rho_j+e_j)) & = 0 \quad \text{ in }  S_\rho\; , 
 \\
\nabla_z\cdot (D^\rho_{u, \text{per}}(z,\omega)\nabla_z \hat \eta^\rho_j - \chi_{\text{per}}^\rho(z,\omega)e_j)&=0   \quad \text{ in } S_\rho \; .
\end{aligned}
 \end{eqnarray}
Given the corrector functions  $\bar \eta^\rho, \,  \, \hat \eta^\rho \in H^1_{\text{per}} (S_\rho)$, the effective coefficients are then defined by
 \begin{eqnarray}
  D^\rho_{\omega, ij} & = & \frac 1 {\rho^d} \int_{S_\rho} \big((D_{u,\text{per}}^\rho(z, \omega) \nabla_{z} \bar \eta^\rho_i)_j+D_{u,\text{per}}^\rho(z, \omega) \delta_{ij}\big)\, dz \;,  \label{per_coef_D} \\
   \chi^\rho_{\omega, ij} &=&- \frac 1 {\rho^d} \int_{S_\rho} \big((D^\rho_{u, \text{per}}(z,\omega)\nabla_{z} \hat \eta^\rho_i)_j - \chi_{\text{per}}^\rho(z,\omega)\delta_{ij} \big)\, dz \; ,
\label{per_coef_chi}
 \end{eqnarray}
for $i,j=1,\ldots, d$, and the macroscopic equations read
\begin{equation*}
\begin{aligned}
 & \partial_t u^\rho = \nabla\cdot(D^\rho_\omega \,  \nabla   u^\rho -  \chi^\rho_\omega \, u^\rho \,  \nabla  v^\rho)&& \text{ in } \, \,  Q_\tau , \\
&  \partial_t v^\rho  =\nabla \cdot  (D_{v}(x)\nabla  v^\rho) - \gamma \,  v^\rho  + \alpha \, u^\rho   &&  \text{ in } \, \,  Q_\tau, \\
& (D^\rho_\omega \,  \nabla   u^\rho -  \chi^\rho_\omega \, u^\rho \,  \nabla  v^\rho)\cdot n =0, \qquad \nabla v^\rho \cdot n =0 \quad &&
\text{ on } \, \, (0,\tau)\times \partial Q 
\end{aligned}
\end{equation*}
for $P$-a.s.\ $\omega \in \Omega$. 

The following theorem is the key result of this section. It guarantees  the convergence of the effective coefficients obtained by periodic approximation to the original  effective coefficients obtained from the stochastic homogenization in the previous section. 

 \begin{theorem}
 Let $D^\rho_\omega$ and $\chi^\rho_\omega$ be the effective coefficients obtained in  \eqref{per_coef_D} and \eqref{per_coef_chi}, respectively.  Then for $D^\ast$ and $\chi^\ast$ as in \eqref{effect_coef}, the following hold true
\begin{equation}
 \lim\limits_{\rho \to \infty}   D^\rho_{\omega, ij}  =  D^\ast_{ij}  \quad \text{P-a.s.}, \qquad
   \lim\limits_{\rho \to \infty}   \chi^\rho_{\omega, ij}  =  \chi^\ast_{ij}  \quad \text{P-a.s.}, \quad i,j=1, \ldots, d.
\end{equation}
 \end{theorem} 
 
 \begin{proof}
First, we  consider   in $S_1=[0,1]$ the auxiliary problems 
 \begin{eqnarray} 
&&  \begin{cases}
 \nabla_x \cdot\big(D^\rho_{u, \text{per}}( \rho x, \omega)(\nabla_x \bar w^\rho_j +e_j)\big) = 0 \hspace{1.7 cm }   \text{ in } \, S_1 \; ,  \\
 \bar w^\rho_j  \qquad S_1-\text{periodic}, \hspace{1.9 cm }    \int_{S_1} \bar w^\rho_j(x)\,  dx =0 \; ,  
 \end{cases}
 \label{w1_rho_period}
 \\
&& \begin{cases}
 \nabla_x \cdot \big(D^\rho_{u, \text{per}}( \rho x, \omega) \nabla_x \hat w^\rho_j - \chi^\rho_{\text{per}}(\rho x, \omega)e_j\big) = 0 \hspace{0.5 cm }  \text{ in } \, S_1 \; ,  \\
  \hat w^\rho_j  \qquad S_1-\text{periodic},  \hspace{1.9 cm }  \int_{S_1} \hat w^\rho_j(x) dx =0\; . 
  \end{cases}
  \label{w2_rho_period}
 \end{eqnarray} 
 
 From the definition of $D^\rho_{u,\text{per}}$ we have that $D^\rho_{u,\text{per}}(\rho x, \omega)= D_u(\rho x, \omega)$ in $S_1$.   Then, for $\rho =1/\ve$ and $Q=S_1$, one can apply the stochastic homogenization results of  Section~\ref{stoch_homog} to  problems  \eqref{w1_rho_period}  and \eqref{w2_rho_period} to obtain the effective macroscopic equations
\begin{eqnarray} \label{macro_aux}
\begin{aligned}
 \nabla_x\cdot (D^\ast(\nabla_x \bar w_j +e_j)) = 0 \quad \text{ in }  S_1 \; , 
\quad
 \bar w_j  \quad   S_1-\text{periodic},  \, \,  \int_{S_1} \bar w_j(x)\,  dx =0\; , \\
 \nabla_x \cdot(D^\ast  \nabla_x \hat w_j - \chi^\ast e_j) = 0 \quad \text{ in }  S_1 \; , 
\quad
  \hat w_j  \quad  S_1-\text{periodic},  \, \,  \int_{S_1} \hat  w_j(x)\,  dx =0 \; , 
 \end{aligned}
 \end{eqnarray} 
 where $j=1,\ldots,d$, and $D^\ast$ and $\chi^\ast$  are given by \eqref{effect_coef}.
 
We then consider  the coordinate transformation  $y=z/\rho$ in equations  \eqref{unit_per_u_v},  transforming $S_\rho$ to the unit cube $S_1$.  We let 
$\bar \eta^\rho_{0,j}(y)= \frac 1 \rho \bar \eta^\rho_j(\rho y)$ and $\hat \eta^\rho_{0,j}(y)= \frac 1 \rho \hat \eta^\rho_j(\rho y)$, and 
   rewrite the equations in \eqref{unit_per_u_v}  as
\begin{eqnarray}
&& \nabla_y \cdot (D^\rho_{u, \text{per}}(\rho y, \omega) ( \nabla_y \bar \eta^\rho_{0,j} + e_j)) = 0   \hspace{2.25 cm }  \text{ in } \, S_1,
\label{scaled_period_micro_1} \\
&& \nabla_y \cdot(D^\rho_{u, \text{per}}(\rho y, \omega)  \nabla_y \hat \eta^\rho_{0,j} -\chi^\rho_{\text{per}}(\rho y, \omega)e_j) = 0  \quad \hspace{0.5 cm }  \text{ in } \, S_1, \label{scaled_period_micro_2}
\end{eqnarray}
where $\bar \eta^\rho_{0,j}$ and $\hat \eta^\rho_{0,j}$ are $S_1$-periodic functions, $j=1,\ldots,d$.   The solutions of  \eqref{scaled_period_micro_1} and \eqref{scaled_period_micro_2} are unique up to  an additive constant, which we  fix by considering 
$\int_{S_1} \bar \eta^\rho_{0,j} (y) dy = 0$ and $\int_{S_1} \hat \eta^\rho_{0,j} (y) dy = 0$. 
Taking $\bar \eta^\rho_0$ and   $\hat \eta^\rho_0$ as test functions in  \eqref{scaled_period_micro_1} and  \eqref{scaled_period_micro_2}, respectively, using Assumption \ref{assumptions} on the coefficients 
$\widetilde D_u$ and $\widetilde \chi$, and applying the Poincar\'e inequality we obtain the following {\it a priori} estimates uniformly in $\rho$
\begin{equation}
\|\bar \eta^\rho_{0,j} \|_{H^1(S_1)} \leq C \; , \quad \|\hat \eta^\rho_{0,j} \|_{H^1(S_1)} \leq C, \quad j=1,\ldots,d \; . 
\end{equation}
Thus, we have that $\bar\eta^\rho_{0,j}$  and  $\hat\eta^\rho_{0,j}$ converge weakly 
 in $H^1_{\text{per}}(S_1)$ to $\bar \eta^\infty_j$ and $\hat \eta^\infty_j$, respectively,  as $\rho\to \infty$, with $j=1,\ldots,d$. We also have that $\bar\eta^\rho_{0,j}$  and  $\hat\eta^\rho_{0,j}$ converge  stochastically two-scale to the same  limit functions $\bar \eta^\infty_j=\bar  \eta^\infty_j(y)$ and $\hat \eta^\infty_j=\hat  \eta^\infty_j(y)$, with $j=1,\ldots,d$.
 Then, considering the results on the stochastic homogenization of equations  \eqref{w1_rho_period} and  \eqref{w2_rho_period},  we obtain that $\bar \eta^\infty_j$ and $\hat \eta^\infty_j$ satisfy 
\begin{eqnarray}
\begin{aligned} \label{limit_eta}
\nabla_y\cdot (D^\ast(\nabla_y \bar \eta^\infty_j+e_j))=0 \quad  \text{ in } S_1 \; , \quad \bar \eta^\infty_j \, \, \,  S_1-\text{periodic}, \,  
 \int_{S_1} \bar \eta^\infty_j (y) dy = 0,  \\
\nabla_y\cdot (D^\ast \nabla_y \hat \eta^\infty_j- \chi^\ast e_j)=0 \quad \text{ in } S_1 \; , \quad \hat \eta^\infty_j \, \, \, S_1-\text{periodic}, \, 
 \int_{S_1} \hat \eta^\infty_j (y) dy = 0  .
\end{aligned}
\end{eqnarray}
Hence, we have that 
\begin{eqnarray}\label{converg_flux}
\quad \begin{aligned} 
D^\rho_{u, \text{per}}(\rho y, \omega) ( \nabla_y \bar \eta^\rho_{0,j} + e_j) &\rightharpoonup& D^\ast(\nabla_y \bar \eta^\infty_j+e_j)  & \, \,  \text{ weakly in  } \, L^2(S_1), \\
 D^\rho_{u, \text{per}}(\rho y, \omega)  \nabla_y \hat \eta^\rho_{0,j} -\chi^\rho_{\text{per}}(\rho y, \omega)e_j & 
 \rightharpoonup &
 D^\ast \nabla_y \hat \eta^\infty_j - \chi^\ast e_j \,& \, \, \text{ weakly in  } \, L^2(S_1), 
 \end{aligned}
\end{eqnarray}
as $\rho \to \infty$,  for $P$-a.s.\  $\omega \in \Omega$ and  $j=1,\ldots,d$.  
Finally, since the only periodic solutions of \eqref{limit_eta} with zero average are   
$\bar \eta^\infty_j (y)=0$
 and  $\hat \eta^\infty_j (y)=0$ for $y \in S_1$,
 it follows from  \eqref{converg_flux} that
\begin{equation*} 
\begin{aligned}
 &D^\rho_{\omega, j} = &&\int_{S_1} D^\rho_{u, \text{per}}(\rho y, \omega) ( \nabla_y \bar \eta^\rho_{0,j} +  e_{j})  \, dy   &\to & 
\int_{S_1} D^\ast e_j\, dy = D^\ast_j\; , \\
& \chi^\rho_{\omega, j} = - &&\int_{S_1} \big(D^\rho_{u, \text{per}}(\rho y, \omega)  \nabla_y \hat \eta^\rho_{0,j} -\chi^\rho_{\text{per}}(\rho y, \omega) e_j\big) \, dy  & \to & \int_{S_1}\chi^\ast e_j  dy =\chi^\ast_j \; ,
\end{aligned}
\end{equation*}
as $\rho \to \infty$,  for $P$-a.s.\  $\omega \in \Omega$ and $j=1,\ldots,d$. This proves  the convergence results stated in the theorem. 
\end{proof}

\section*{Acknowledgments}
The authors would like to thank Prof.\ Andrey Piatnitski for advice and encouragement. AM would like to thank the Computational Science and Engineering Laboratory at ETH Z\"urich for the warm hospitality during a sabbatical semester.
The research of AM is supported in part by the National Science Foundation under Grants NSF CDS\&E-MSS 1521266 and NSF CAREER 1552903. The research of MP is supported in part by the EPSRC First Grant EP/K036521/1. \\

\end{document}